\documentclass{article}

\usepackage{arxiv}

\usepackage[utf8]{inputenc} 
\usepackage[T1]{fontenc}    
\usepackage{hyperref}       
\usepackage{url}            
\usepackage{booktabs}       
\usepackage{amsfonts}       
\usepackage{nicefrac}       
\usepackage{microtype}      
\usepackage{lipsum}		
\usepackage{graphicx}

\usepackage{doi}

\usepackage[section]{placeins}

\usepackage{textalpha}
\usepackage{amssymb}

\usepackage{longtable}
\usepackage[ruled,vlined]{algorithm2e}
\usepackage{pgfplots}
\pgfplotsset{compat=newest}
\usepgfplotslibrary{fillbetween}
\usepackage{textalpha}
\usepackage{xcolor}

\usepackage[T1]{fontenc} 

\usepackage[utf8]{inputenc} 

\usepackage{graphicx} 
\graphicspath{{Figures/}} 

\usepackage{enumitem} 

\usepackage{lipsum} 

\usepackage{subfig} 

\usepackage{amsmath,amsthm} 

\usepackage{varioref} 

\usepackage{tikz}
\usetikzlibrary{3d}
\usetikzlibrary{arrows,shapes,positioning,calc}
\theoremstyle{definition} 

\theoremstyle{plain} 
\newtheorem{theorem}{Theorem}

\theoremstyle{remark} 
\newtheorem*{remark}{Remark}

\newcommand{\R}{\mathbb{R}}
\newcommand{\usol}{u}
\newcommand{\uinc}{u_{inc}}
\newcommand{\usc}{u_{sc}}

\newcommand{\x}{x}
\newcommand{\azero}{\alpha_{0}}
\newcommand{\aone}{\alpha_{1}}
\newcommand{\J}{J}
\newcommand{\baru}{\bar{\usol}}

\newcommand{\eps}{\epsilon}
\newcommand{\Omeps}{\Omega_{\eps}}
\newcommand{\Om}{\Omega}
\newcommand{\vt}{v}

\newcommand{\nv}{\vec{n}}
\newcommand{\Vv}{\vec{V}}
\newcommand{\Wv}{\vec{W}}

\newcommand{\nab}{\nabla}

\newcommand{\Ga}{\Gamma_1}
\newcommand{\Gb}{\Gamma_2}
\newcommand{\Gc}{\Gamma_3}
\newcommand{\Gd}{\Gamma_4}
\newcommand{\Ge}{\Gamma_5}
\newcommand{\Lag}{\mathcal{L}}

\newcommand*\diff{\mathop{}\!\mathrm{d}}

\newcommand\restr[2]{{
		\left.\kern-\nulldelimiterspace 
		#1 
		\vphantom{\big|} 
		\right|_{#2} 
}}

\title{Shape Optimization for the Mitigation of Coastal Erosion via the Helmholtz Equation}

\author{ \href{https://orcid.org/0000-0001-9930-065X}{\includegraphics[scale=0.06]{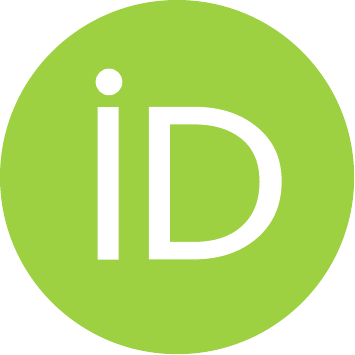}\hspace{1mm}Luka Schlegel} \\
	Department of Mathematics\\
	Universität Trier\\
	Universitätsring 15, 54296 Trier\\
	\texttt{schlegel@uni-trier.de} \\
	\And
	\href{https://orcid.org/0000-0001-7665-130X}{\includegraphics[scale=0.06]{orcid.pdf}\hspace{1mm}Volker Schulz} \\
	Department of Mathematics\\
	Universität Trier\\
	Universitätsring 15, 54296 Trier\\
	\texttt{volker.schulz@uni-trier.de} \\
}

\renewcommand{\headeright}{}
\renewcommand{\undertitle}{}
\renewcommand{\shorttitle}{SOMICE}

\hypersetup{
pdftitle={Helmholtz Paper},
pdfsubject={q-bio.NC, q-bio.QM},
pdfauthor={Luka Schlegel, Volker Schulz},
pdfkeywords={Coastal Erosion, Shape Optimization, Helmholtz Equation},
}

\begin{document}
\maketitle

\begin{abstract}
	Coastal erosion describes the displacement of land caused by destructive sea waves, currents or tides. Major efforts have been made to mitigate these effects using groins, breakwaters and various other structures. We try to address this problem by applying shape optimization techniques on the obstacles. A first approach models the propagation of waves towards the coastline, using a 2D time-harmonic system based on the famous Helmholtz equation in the form of a scattering problem. The obstacle's shape is optimized over an appropriate cost function to minimize the height of water waves along the shoreline, without relying on a finite-dimensional design space, but based on shape calculus. 

\end{abstract}


\section{Introduction}
Coastal erosion describes the displacement of land caused by destructive sea waves, currents or tides. Major efforts have been made to mitigate these effects using groins, breakwaters and various other structures. 
Among experimental set-ups to model the propagation of waves towards a shore and to find optimal wave-breaking obstacles, the focus has turned towards numerical simulations due to the continuously increasing computational performance. Essential contributions to the field of numerical coastal protection have been made for steady \cite{Azerad2005}\cite{Mohammadi2008}\cite{Keuthen2015} and unsteady \cite{Mohammadi2011}\cite{Mohammadi2012} descriptions of propagating waves. In this paper we select the Helmholtz Equation, that originates from the wave equation via separation of variables and assuming time independence. This paper builds up on the monographs \cite{Choi1987}\cite{Sokolowski1992}\cite{Delfour2011} to perform free-form shape optimization. In addition we strongly orientate on \cite{Schulz2014b}\cite{Schulz2016}\cite{Schulz2016a} that use the Lagrangian approach for shape optimization, i.e. calculating state, adjoint and the deformation of the mesh via the volume form of the shape derivative assembled on the right-hand-side of the linear elasticity equation, as Riesz representative of the shape derivative. The application of shape-calculus-based shape optimization to prevent coastal erosion by optimizing the form of the Helmholtz scatterer builds an extension to \cite{Azerad2005}, who relied on a fixed parametrization and to \cite{Keuthen2015}, who used a level set method for shape optimization. The paper is structured as follows: In Section \ref{HelmModForm} we formulate the PDE-constrained optimization problem. In Section \ref{sec:HelmAdjShapOpt} we will derive the necessary tools to solve this problem, by deriving the adjoint equation, the shape derivative in volume form and boundary form such as the topological derivative. The final part, Section \ref{sec:NumRes}, will then apply the results to firstly a simplified mesh and secondly to a more realistic mesh, picturing the Langue-de-Barbarie, a coastal section in the north of Dakar, Senegal that was severely affected by coastal erosion within the last decades.

\section{Model Formulation}
\label{HelmModForm}
Suppose we are given an open domain $\tilde{\Om}\subset\R^2$, which is split into the disjoint sets $\Om,D\subset\tilde{\Om}$ such that $\Om\cup D\cup\Ge=\tilde{\Om}$, $\Ga\cup\Gb\cup\Gc\cup\Gd=\partial\tilde{\Om}$ and $\Ga\cup\Gb\cup\Gc\cup\Gd\cup\Ge:=\Gamma$. In this setting we assume the variable, interior $\Ge$ and the fixed outer boundary  $\partial\tilde{\Om}$ to be at least Lipschitz. One simple example of such kind is visualized below in Figure \ref{fig:DomainHelmholtz}. 

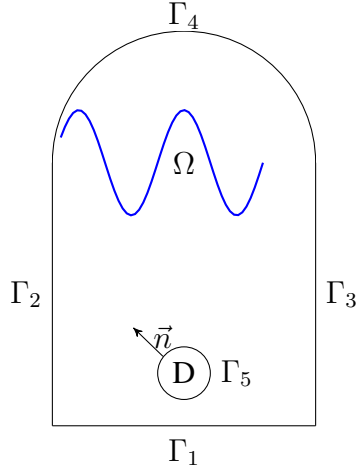
\begin{figure}[h]
	\centering
	\begin{tikzpicture}[scale=0.7]
	\begin{scope}
	\clip (0,0) rectangle (5,2.5);
	\draw (2.5,0) circle(2.5);
	\end{scope}	
	
	\draw (0,-5) --(0,0);
	\draw (0,-5) --(5,-5);
	\draw (5,-5) --(5,0);
	\draw[->, >=stealth', shorten >=1pt] (2.1,-3.68)   -- (1.5,-3.1);
	\draw (2.5,-4) circle (0.5);
	\tiny\draw[thick, blue] plot[domain=0.5:4*pi, samples=50]  (\x/pi,{sin(\x r)});
	
	\node (A) at (2.5,0) {\large $\Omega$};
	\node (B) at (2.5,-5.5) {\large $\Ga$};
	\node (C) at (-0.5,-2.5) {\large $\Gb$};
	\node (D) at (5.5,-2.5) {\large $\Gc$};
	\node (E) at (2.5,2.8) {\large $\Gd$};
	\node (F) at (3.5,-4) {\large $\Ge$};
	\node (G) at (2.5,-4) {\large D};
	\node (H) at (2.1,-3.3) {\large$\nv$};
	\end{tikzpicture}
	\caption[Illustrative Domain]{Illustrative Domain with Initial Circled Obstacle}
	\label{fig:DomainHelmholtz}
\end{figure}
We interpret $\Ga$ as coastline, $\Gb$ and $\Gc$ as lateral sea, $\Gd$ as open sea and $\Ge$ as obstacle boundary. 
The PDE-constrained optimization problem on this domain is defined as \cite{Azerad2005} 

\begin{align}
\min J(\Om) \label{Objective}\\
\textrm{s.t.} \quad
&&-\nabla^2\usol-k^2\usol&=0	&\text{on }\quad &\Om \label{Helmholtz}&&\\
&&\frac{\partial \usol}{\partial \nv}+k\alpha u &= 0	 &\text{on } \quad&\Ga, \Ge&& \label{PABC}\\ 
&&\usol_{|_{\Gb}}&=\usol_{|_{\Gc}}	&\text{on }\quad& \Gb, \Gc &&\label{PBC}\\
&&\frac{\partial \usol-\uinc}{\partial \nv}-ik\left(\usol-\uinc\right) &= 0	&\text{on }\quad& \Gd&& \label{Sommerfeld}
\end{align}

\begin{remark}
	The PDE constrained optimization problem is mainly taken from \cite{Azerad2005}, however we will tackle the problem by not relying on a finite design space but on shape calculus.
\end{remark}

In the following subsections, we will shortly elaborate on the components of (\ref{Objective})-(\ref{Sommerfeld}).

\subsection{Wave Description}
On the illustrative domain we intend to model water waves by complex solution field $u:\Om\rightarrow \mathbb{C}$ to the stationary elliptic Helmholtz equation, i.e.
\begin{align}
-\nabla^2\usol-k^2\usol=0
\end{align}
The complexity is introduced for a total field consisting of $\usol=u_{inc}+u_{sc}$, since the incoming wave is defined as $\uinc(x)=A \exp(ik\x*d_{\phi})$, where $k>0$ is a constant wavenumber, $A>0$ the amplitude or maximal surface elevation and $d_{\phi}$ the wave direction with $d_{\phi}=(\cos\phi,\sin\phi)$ for $\phi\in\R$.\\
In the course of this chapter we will also deal with a second problem, placing a transmissive obstacle $D$ in $\tilde{\Om}$ for porosity coefficient $\phi\in(0,1]$. For this we firstly modify the problem such that we are solving for two distinct wave fields on $\Om$ and $D$ (cf. to \cite{Yan1997}), i.e.
\begin{align}
-\nabla^2\usol-k^2\usol&=0 &\text{on} \quad \Om &&\label{HelmholtzOm}\\
-\nabla^2s-k^2s&=0 &\text{on} \quad D &&\label{HelmholtzD}
\end{align}
for transmission boundaries 
\begin{equation}
\begin{aligned}
\usol&=s&\text{on}\quad \Ge\\
\frac{\partial \usol}{\partial\nv}&=\phi\frac{\partial s}{\nv}&\text{on}\quad \Ge
\label{Eq:HelmholtztransmissiveBC2}
\end{aligned}
\end{equation}
for normal vector $\nv$.
The two wave fields can be rewritten by the usage of a discontinuous transmission coefficient $\phi$
\begin{align}
\phi:=\begin{cases}
\phi_1,&\text{ if }x\in\Om\\
\phi_2,&\text{ if }x\in D
\end{cases}
\end{align}
with interface boundary conditions in the sense of (\ref{Eq:HelmholtztransmissiveBC2}) as
\begin{equation}
\begin{aligned}
\relax
[\![u]\!]&=0\\
[\![\phi\frac{\partial u}{\partial \nv}]\!]&=0
\end{aligned}
\end{equation}
\begin{remark}
	The latter approach requires calculations on $\tilde{\Om}$, such that we can write integrals as
	\begin{align}
	\int_{\tilde{\Om}}=\int_\Om+\int_D
	\label{Eq:RewriteIntegrals}
	\end{align}
\end{remark}
\begin{remark}
	The obstacles can in the transmissive case be interpreted as permeable with regards to propagating waves, e.g. in \cite{Mohammadi2008} geotextile tubes are proposed.
\end{remark}
\begin{remark}
	The derivation of adjoint and shape derivative in Section \ref{sec:HelmAdjShapOpt} is based on (\ref{Objective})-(\ref{Sommerfeld}). The transmissive case, follows analogously by rewriting integrals according to (\ref{Eq:RewriteIntegrals}) with inclusion of transmissive boundaries. conditions.
\end{remark}
\begin{remark}
    Choosing $\phi_1=\phi_2=1$ let us solve the classical Helmholtz equation on the whole domain.
\end{remark}

\subsection{Periodic Boundary Condition}
Periodic boundary conditions are used as
\begin{align}
\usol_{|_{\Gb}}=\usol_{|_{\Gc}}
\end{align}
\begin{remark}
	These conditions allow to significantly reduce the computational domain size, e.g. instead of modelling the whole coastline, it allows the field calculation for a reduced domain assuming periodic reproducibility along the shore (cf. exemplifying to Figure \ref{Fig:PBC}).
	\begin{figure}[h]
		\centering
		\includegraphics[scale=0.75]{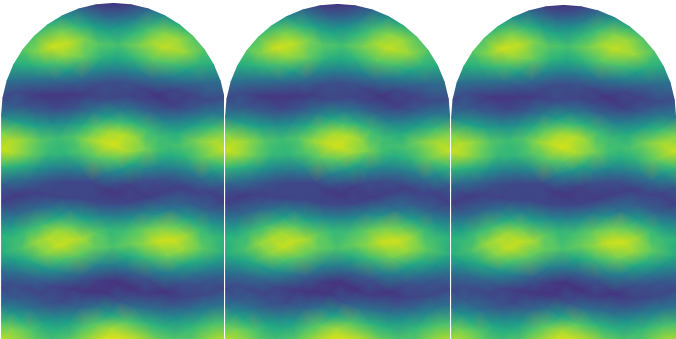}
		\caption[Usage of PBC]{Usage of Periodic Boundary Conditions}
		\label{Fig:PBC}
	\end{figure}
\end{remark}

\subsection{Sommerfeld Radiation Condition}
The model with is equipped with open or non-reflecting boundary at sea level in form of the Sommerfeld Radiation Condition \cite{Schot1992}. It ensures the uniqueness of the solution by demanding that scattered waves are not reflected at an infinite boundary by requiring
\begin{align}
\lim\limits_{r\rightarrow\infty} r^{1/2}\left(\frac{\partial \usc}{\partial r}-ik\usc\right) &= 0 \\ r&=|\x|
\end{align}
Since we are restricted to a finite domain, we adjust the condition using the first order approximation \cite{Azerad2005} on $\Gd$
\begin{align}
\frac{\partial \usol-\uinc}{\partial \nv}-ik(\usol-\uinc) = 0
\end{align}
where due to the half circled boundary $\frac{\partial}{\partial r}=\frac{\partial}{\partial \nv}$ for radial distance $r$. 
\begin{remark}
	For a more comprehensive view on wave propagation a second order approximation \cite{Shirron1998} or a "Perfectly Matched Layer" \cite{Becache2004} could be used. However, latter needs particular attention in a shape optimization routine.
\end{remark}

\subsection{Partially Absorbing Boundary Condition}
We assume partial reflection of the waves at the coast and at the obstacle, by partial absorbing boundary condition (\ref{PABC}) on $\Ga$ and $\Ge$, i.e.
\begin{align}
\frac{\partial \usol}{\partial \nv}+k\alpha u = 0	 
\end{align}
where $\alpha=\azero+i\aone$ represents the complex transmission coefficient as introduced in \cite{Berkhoff1976}. 
\begin{remark}
	The general solution was derived by Berkhoff \cite{Berkhoff1976} for $\alpha$ as
	\begin{equation}
	\begin{aligned}
	\azero = \frac{2K\sin\beta\cos\gamma}{1+K^2+2K\cos\beta}\\
	\aone = \frac{(1-K^2)\cos\gamma}{1+K^2+2K\cos\beta}	
	\end{aligned}
	\label{Eq:BerkhoffSol}
	\end{equation}
	with reflection coefficient $K$, reflection phase angle $\beta$ and the incident wave direction $\gamma$. 
	
	\begin{remark}
		The choice of $\alpha$ is a priori a rocky question as it needs to incorporate the angle of the incoming such as already already scattered waves.
		Since the preceding formulation is based on the assumption that the transmission coefficient is known for all parts of the boundary, we follow \cite{Isaacson1990} for simplification and set $\gamma=0$ and $\beta=0$, which leads for (\ref{Eq:BerkhoffSol}) to: 
		\begin{equation}
		\begin{aligned}
		\azero&=0\\
		\aone&=\frac{1-K}{1+K}
		\end{aligned}
		\label{Eq:Isaacsonalpha}
		\end{equation}
		
	\end{remark}	
	
\end{remark}
\begin{remark}
	In \cite{Azerad2005} and \cite{Keuthen2015} the simple case for $\alpha=0$ is considered such that (\ref{PABC}) reduces to $\frac{\partial \usol}{\partial n}=0$ which is commonly referred to as sound-hard scattering \cite{Colton1992} implying
	\begin{align}
	\frac{\partial \uinc}{\partial \nv}=-\frac{\partial \usc}{\partial \nv}
	\end{align}
	This assumption simplifies not only the calculation for the field but also for the shape derivative as we will see in Section \ref{sec:HelmAdjShapOpt}. However, this may lead to undesired reflections at respective boundaries.
\end{remark}
\begin{remark}
	Frequently obstacle problems with Dirichlet boundary conditions are investigated, which is referred to as sound-soft scattering \cite{Colton1992}. We will restrict ourself to the sound-hard case and otherwise refer to \cite{Feijoo2003}.
\end{remark}

\subsection{Objective Function}
An obstacle is considered to be optimal, if the squared difference of wave $u$ and target height $u_0$ is minimized and the distribution is as uniform as possible along the shore. Hence, we define the objective $J:\Om\rightarrow\R$
\begin{enumerate}
	\item for single directions and frequencies as:
	\begin{align}
	\J_1(\Om)=&||u-\tilde{u}||_{L^2_{\R}(\Ga)}^2+\xi||u-\baru_{\Ga}||_{L^2_{\R}(\Ga)}^2 \label{Eq:ObjectiveSingle}
	\end{align}
	for target height $\tilde{u}$ such as variance weight $\xi$ and mean elevation $\baru_{\Ga}=1/l\int_{\Ga}u\diff s$.
	\item  for multiple directions $(\phi_j)_{1\leq j \leq N}$ with different weights $w_j$ each for different wave numbers $(k_i)_{1\leq i \leq M}$ as:
	\begin{align}
	J_2(\Om)=\sum_{i,j}^{M,N}w_j\J_{k_i,\phi_j}(\Om)
	\label{Eq:ObjectiveSum}
	\end{align}	
\end{enumerate}
In either case to ensure that the obstacle is not becoming arbitrarily large we add a volume penalty controlled by $\nu_1\geq0$, i.e.
\begin{align}
J_3=\nu_1\int_{\Om}1\diff x
\end{align}
and a perimeter regularization, controlled by the parameter $\nu_2\geq0$ as
\begin{align}
J_4=\nu_2\int_{\Ge}1\diff\Ge
\end{align}
\begin{remark}
	In (\ref{Eq:ObjectiveSingle}) we follow \cite{Keuthen2015} and define
	\begin{align}
	(f,g)_{L^2_{\R}(A)}:=\Re(f,g)_{L^2_{\mathbb{C}(A)}}:=\Re(f,\bar{g})_{L^2(A)} \label{Eq:CompProduct}
	\end{align}
\end{remark}
\begin{remark}
	Defining the inner product as in (\ref{Eq:CompProduct}) is beneficial for numerical implementations as described in Section \ref{sec:NumRes}, since $L^2_{\R}(A)$ would lead to mixed spaces that are cumbersome to model.
\end{remark}

\section{Adjoint-Based Shape \& Topology Optimization}
\label{sec:HelmAdjShapOpt}
In this section we will first introduce the necessary tools for adjoint-based shape optimization, before we apply techniques to the PDE-constrained problem from the previous section.
\subsection{Notations and Definitions}
The idea of shape optimization is to deform an object ideally to minimize some target functional. Hence, to find a suitable matter of deforming we are interested in some shape analogy of a classical derivative. Here we use a methodology that is commonly used in shape optimization and extensively elaborated in various works \cite{Choi1987}\cite{Sokolowski1992}\cite{Delfour2011}.\\
In this section we fix notations and definitions following \cite{Schulz2016}\cite{Schulz2016a} and amend whenever it appears necessary.
We start by introducing a family of mappings $\{\phi_\eps\}_{\eps\in[0,\tau]}$ for $\tau>0$ that are used to map each current position $\x\in\Om$ to another by $\phi_\eps(\x)$, where we choose the vector field $\Vv$ as the direction for the so called perturbation of identity
\begin{align}
\x_\eps=\phi_\eps(\x)=\x+\eps \Vv(\x)
\label{Eq:4poi}
\end{align}
According to this methodology, we can map the whole domain $\Om$ to another $\Omeps$ such that 
\begin{align}
\Omeps=\phi_\eps(\Om)=\{\x_\eps|x+\eps \Vv(x),x\in\Om\}
\label{Eq:5domain}
\end{align}
Minimization of a generic functional dependent on the domain $J:\Om\rightarrow\R$ often requires the derivatives. Hence, we define the Eulerian Derivative as 
\begin{align}
DJ(\Om)[\Vv]=lim_{\eps\rightarrow 0^+} \frac{J(\Om_\eps)-J(\Om)}{\eps}\label{Eq:7EulerDer}
\end{align}
Commonly, this expression is called shape derivative of $J$ at $\Om$ in direction $\Vv$ and in this sense $J$ shape differentiable at $\Om$ if for all directions $\Vv$ the Eulerian derivative exists and the mapping $\Vv\mapsto DJ(\Om)[\Vv]$ is linear and continuous.
In addition, we define the material derivative of some scalar function $p:\Om_\eps\rightarrow\R$ at $x\in\Om$ with respect to the deformation $\phi_\eps$ as
\begin{align}
	D_mp(x):=lim_{\eps\rightarrow0^+}\frac{p\circ \phi_\eps(x)-p(x)}{\eps}=\frac{d^+}{d\eps}\restr{(p\circ \phi_\eps)(x)}{\eps=0}\label{Eq:8MatDer}
\end{align} 
and the corresponding shape derivative for a scalar $p$.
In the following, we will use the abbreviation $\dot{p}$ to mark the material derivative of $p$. In Section \ref{sec:HelmAdjShapOpt} we will need to have the following calculation rules on board \cite{Berggren2010}
\begin{align}
	D_m(pq)&=D_mpq+pD_mq\label{Eq:10MatProdR}\\
	D_m\nab p&=\nab D_mp-\nab \Vv^T\nab p\label{Eq:11MatGradR}\\
	D_m(\nab q^T\nab p)&=\nab D_mp^T\nab q-\nab q^T(\nab \Vv+\nab \Vv^T)\nab p+\nab p^T\nab D_mq \label{Eq:12MatGradProdR}
\end{align}
The basic idea in the proof of the shape derivative in the next section will be to pull back each integral defined on the on the transformed field back to the original configuration. We therefore need to state the following rule for differentiating domain integrals \cite{Berggren2010}.
\begin{align}
	\frac{d^+}{d\eps}\restr{\left(\int_{\Om_{\eps}}p(\eps)\right)}{\eps=0}=\int_\Om(D_mp+\nab\cdot \Vv p) \label{Eq:13DoaminR}
\end{align}

\subsection{Adjoint-Based Shape \& Topology Optimization}

We reformulate the constrained optimization problem (\ref{Objective})-(\ref{Sommerfeld}) with the help of the Lagrangian
\begin{align}
\Lag(\Om,\usol,\vt,\vt_1,\vt_2,\vt_3)=\J_1(\Om)+a(\Om;\usol,\vt,\vt_1,\vt_2,\vt_3)-l(\Om;\vt_3)
\label{Eq:LagrangianHelmholtz}
\end{align}
where $J_1$ is the objective (\ref{Eq:ObjectiveSingle}), $a(\Om;\usol,\vt,\vt_1,\vt_2,\vt_3)$ is the bilinear form obtained from boundary value problem (\ref{Helmholtz})-(\ref{Sommerfeld}) and $v_1,v_2,v_3\in H^1(\Om)$ are multipliers.
\begin{equation}
\begin{aligned}
a(\Om;\usol,\vt,\vt_1,\vt_2,\vt_3)=&\big(\nab\usol,\nab\vt\big)_{L^2_{\R}(\Om)}- k^2\big(\usol,\vt\big)_{L^2_{\R}(\Om)}-\big(\frac{\partial\usol}{\partial\nv},\vt\big)_{L^2_{\R}(\Gamma)}-\\
&\big(\frac{\partial \usol}{\partial \nv}+k\alpha\usol,\vt_1\big)_{L^2_{\R}(\Ga,\Ge)}+\big(u,\vt_2\big)_{L^2_{\R}(\Gb)}+(u,\vt_2)_{L^2_{\R}(\Gc)}-\\
&\big(\frac{\partial \usol}{\partial \nv}-ik(\usol),\vt_3\big)_{L^2_{\R}(\Gd)}
\label{Eq:WeakForm}
\end{aligned}
\end{equation}
and $l(\Om;\vt_3)$ is the bilinear form defined by
\begin{equation}
\begin{aligned}
l(\Om,\vt_3)=&\big(\frac{\partial \uinc}{\partial \nv}-ik(\uinc),\vt_3\big)_{L^2_{\R}(\Gd)}
\end{aligned}
\end{equation}
\begin{remark}
	The transmissive case leads us to weak forms as
	\begin{equation}
	\begin{aligned}
	a(\Om;\usol,\vt,\vt_1,\vt_2,\vt_3)=&\big(\phi\nab\usol,\nab\vt\big)_{L^2_{\R}(\Om)}- k^2\big(\phi\usol,\vt\big)_{L^2_{\R}(\Om)}-\big(\phi\frac{\partial\usol}{\partial\nv},\vt\big)_{L^2_{\R}(\partial\tilde{\Om})}-\\
	&\big(\phi\frac{\partial\usol}{\partial\nv},\vt\big)_{L^2_{\R,[\![]\!]}(\Ge)}-
	\big(\phi\frac{\partial \usol}{\partial \nv}-k\alpha\usol,\vt_2\big)_{L^2_{\R}(\Ga)}-(u,\vt_2)_{L^2_{\R}(\Gb)}-\\
	&\big(u,\vt_2\big)_{L^2_{\R}(\Gc)}-
	\big(\frac{\partial \usol}{\partial \nv}-ik(\usol),\vt_3\big)_{L^2_{\R}(\Gd)}
	\label{Eq:WeakFormTrans}
	\end{aligned}
	\end{equation}
	and $l(\Om;\vt_3)$ is the bilinear form defined by
	\begin{equation}
	\begin{aligned}
	l(\Om,\vt_3)=&\big(\frac{\partial \uinc}{\partial \nv}-ik(\uinc),\vt_3\big)_{L^2_{\R}(\Gd)}
	\end{aligned}
	\end{equation}
	where ${L^2_{\R,[\![]\!]}(\Ge)}$ denotes the usage of jumps in the associated integrals.
\end{remark}
\begin{remark}
	To continue with adjoint calculations we are required to integrate twice by parts on the derivative-containing terms such that we obtain exemplifying obtain for non-transmissive obstacle
	\begin{equation}
	\begin{aligned}
	a(\Om;\vt,\usol)=&-\big(\usol,\nab^2\vt\big)_{L^2_{\R}(\Om)}-\big(\vt,\frac{\partial\usol }{\partial\nv}\big)_{L^2_{\R}(\Gamma)}+\big(\usol,\frac{\partial\vt}{\partial\nv}\big)_{L^2_{\R}(\Gamma)}-k^2\big(\usol,\vt\big)_{L^2_{\R}(\Om)}-\\
	&\big(\frac{\partial \usol}{\partial \nv}+k\alpha\usol,\vt_1\big)_{L^2_{\R}(\Ga,\Ge)}-\big(u,\vt_2\big)_{L^2_{\R}(\Gb)}-\big(u,\vt_2\big)_{L^2_{\R}(\Gc)}-\\
	&\big(\frac{\partial \usol}{\partial \nv}-ik(\usol),\vt_3\big)_{L^2_{\R}(\Gd)}
	\end{aligned}
	\end{equation}
\end{remark}
\begin{remark}
	Instead of multipliers it is also possible to derive adjoint and shape derivative based on an immediate insertion of boundary conditions.
\end{remark}
\begin{remark}
	We can regard the Lagrangian (\ref{Eq:LagrangianHelmholtz}) w.r.t. $J_2$ in the same manner. In the following we restrict to $J_1$ for readability.
\end{remark}

We obtain the state equation from differentiating the Lagrangian for $\vt$ and  the adjoint equation from differentiating w.r.t. $\usol$. As in \cite{Schulz2014b}, the theorem of Correa and Seger \cite{Correa1985} is applied on the right hand side of (\ref{Eq:191SaddleHelm}) so that the following equality holds 
\begin{align}
J_1(\Om)=\min_{\usol}\max_{\vt}\Lag(\Om,\usol,\vt)
\label{Eq:191SaddleHelm}
\end{align}

The adjoint is formulated in the following theorem:  
\begin{theorem}
	(Adjoint) Assume that the elliptic PDE problem (\ref{Helmholtz})-(\ref{Sommerfeld}) is $H^1$-regular and $\alpha$ as in (\ref{Eq:Isaacsonalpha}), so that its solution $\usol$ is at least in $H^1(\Om)$. Then the adjoint in strong form (without perimeter regularization and variance penalty) is given by
	
	\begin{equation}
	\begin{aligned}
	-\nab^2v-k^2v&=0 \quad &\text{ on }&\quad \Om&&\\
	\text{ s.t. }\frac{\partial v}{\partial \nv}&=-(\usol-u_0)+k\alpha v &\text{ on } &\quad\Ga&&\\
	\frac{\partial v}{\partial \nv}&=k\alpha v &\text{ on }& \quad\Ge&&\\
	\frac{\partial v}{\partial \nv}&=-ik v &\text{ on }&\quad \Gd&&\\
	\vt&=0 &\text{ on }&\quad \Gb,\Gc&&
	\end{aligned}
	\label{Eq:19Adjoint}
	\end{equation}
\end{theorem}

\begin{proof}
	Any directional derivative of $\mathcal{L}$ w.r.t. $\tilde{\usol}$ must be zero at the solution $\usol$, hence
	\begin{equation}
	\begin{aligned}
	0=&\frac{d}{d\eps}\restr{\mathcal{L}(\usol +\eps\tilde{\usol},\vt,\vt_1,\vt_2,\vt_2,\vt_3)}{\eps=0}\\
	=&\frac{d}{d\eps}	\Big[\frac{1}{2}\big(\usol+\eps\tilde{\usol}-u_0,\usol+\eps\tilde{\usol}-u_0\big)_{L^2_{\R}(\Ga)}-\big(\usol+\eps\tilde{\usol},\nab^2\vt\big)_{L^2_{\R}(\Om)}-\\
	&\big(\vt,\frac{\partial\usol+\eps\tilde{\usol}}{\partial\nv}\big)_{L^2_{\R}(\Gamma)}+\big(\usol+\eps\tilde{\usol},\frac{\partial\vt}{\partial\nv}\big)_{L^2_{\R}(\Gamma)}-k^2\big(\usol+\eps\tilde{\usol},\vt\big)_{L^2_{\R}(\Om)}-\\
	&\big(\frac{\partial \big(\usol+\eps\tilde{\usol})-\uinc}{\partial \nv}-ik(\usol+\eps\tilde{\usol}-u_{inc}),\vt_3\big)_{L^2_{\R}(\Gd)}-\\
	&\big(\frac{\partial (\usol+\eps\tilde{\usol})}{\partial \nv}+k\alpha(\usol+\eps\tilde{\usol}),\vt_1\big)_{L^2_{\R}(\Ga,\Ge)}+\\		&\restr{\big(u+\eps\tilde{\usol},\vt_2\big)_{L^2_{\R}(\Gb)}+\big(u+\eps\tilde{\usol},\vt_2\big)_{L^2_{\R}(\Gc)}\Big]}{\eps=0}\\
	=&\big(\usol-u_0,\tilde{\usol}\big)_{L^2_{\R}(\Ga)}-\big(\tilde{\usol},\nab^2\vt\big)_{L^2_{\R}(\Om)}-\\
	&\big(\vt,\frac{\partial\tilde{\usol}}{\partial\nv}\big)_{L^2_{\R}(\Gamma)}+\big(\tilde{\usol},\frac{\partial\vt}{\partial\nv}\big)_{L^2_{\R}(\Gamma)}-k^2\big(\tilde{\usol},\vt\big)_{L^2_{\R}(\Om)}-\\
	&\big(\frac{ (\partial\tilde{\usol})}{\partial \nv}-ik(\tilde{\usol}),\vt_3\big)_{L^2_{\R}(\Gd)}-(\frac{\partial (\tilde{\usol})}{\partial \nv}+k\alpha(\tilde{\usol}),\vt_1\big)_{L^2_{\R}(\Ga,\Ge)}-\\		&\big(\tilde{\usol},\vt_2\big)_{L^2_{\R}(\Gb)}-\big(\tilde{\usol},\vt_2\big)_{L^2_{\R}(\Gc)}
	\end{aligned}
	\end{equation}
	From this we get the adjoint in strong form for the domain $\Om$ by taking the variation as $\tilde{u}\in C_0^\infty$, we obtain
	\begin{equation}
	\begin{aligned}
	-\nab^2v-k^2v&=0 \quad \text{ on } \Om\\
	\end{aligned}
	\end{equation}
	In addition $\tilde{u}\in H^1_0(\Om)$ leads to
	\begin{align}
	-\big(\vt,\frac{\partial\tilde{\usol}}{\partial\nv}\big)_{L^2_{\R}(\Gamma)}-\big(\vt_3,\frac{\partial\tilde{\usol}}{\partial\nv}\big)_{L^2_{\R}(\Gd)}-\big(\vt_1,\frac{\partial\tilde{\usol}}{\partial\nv}\big)_{L^2_{\R}(\Ga,\Ge)}=0
	\end{align}
	From this we know that $\vt=0$ on $\Gb$ and $\Gc$, such as $\vt=-v_3$  on $\Gd$ and  $\vt=-v_1$ on $\Ga$ and $\Ge$. With this, $\tilde{u}\in H^1(\Om)$ leads to
	\begin{equation}
	\begin{aligned}
	&\big(\usol-u_0,\tilde{\usol}\big)_{L^2_{\R}(\Ga)}+\big(\tilde{\usol},\frac{\partial\vt}{\partial\nv}\big)_{L^2_{\R}(\Gamma)}+\\
	&\big(ik(\tilde{\usol}),\vt_3\big)_{L^2_{\R}(\Gd)}-\big(k\alpha(\tilde{\usol}),\vt\big)_{L^2_{\R}(\Ga,\Ge)}+\\		&\big(\tilde{\usol},\vt_2\big)_{L^2_{\R}(\Gb)}+\big(\tilde{\usol},\vt_2\big)_{L^2_{\R}(\Gc)}=0
	\end{aligned}
	\end{equation}
	Which provides us with boundary conditions for normal $\nv$ as claimed, due to the complex  symmetry of the complex inner product.
\end{proof}

Having computed both $\vt$ and $\usol$ we can go over and compute the shape derivative (\ref{Eq:7EulerDer}). 
\begin{remark}
	Shape derivatives can for a sufficiently smooth domain be described via boundary formulations using Hadamard's structure theorem \cite{Sokolowski1992}. The integral over $\Om$ is then replaced by an integral over $\Ge$ that acts on the associated normal vector. In this paper we calculate the deformation field based on the domain formulation and use the boundary formulation for the topological derivative.
\end{remark}

\begin{theorem} \label{Theo:2}
	(Shape Derivative Volume Form)
	Assume that the elliptic PDE problem (\ref{Helmholtz})-(\ref{Sommerfeld}) is $H^1$-regular, so that its solution $\usol$ is at least in $H^1(\Om)$. Moreover, assume that
	the adjoint equation (\ref{Eq:19Adjoint}) admits a solution $\vt\in H^1(\Om)$. Then the
	shape derivative of the objective $J_1$ (without perimeter regularization for full-reflecting boundaries) at $\Om$ in the
	direction $\Vv$ is given by
	
	\begin{equation}
	\begin{aligned}
	DJ_{1,\Om}(\Om)[\Vv]=&\Re\Big[\int_\Om\left(\nab v\cdot\nab\usol-k^2v\usol\right)div \Vv\diff x\\
	&-\int_\Om\nab v\cdot(\nab \Vv+\nab \Vv^T)\nab\usol\diff x\Big]
	\end{aligned}
	\label{Eq:21SD}
	\end{equation}
\end{theorem}

\begin{proof}
	The basic idea is to pull all expressions back to the original configuration \cite{Feijoo2003}. For readability we analyse each term by its own. We note that all terms, including the objective, solely dependent on boundaries other than $\Ge$ vanish, since these are defined to be invariant under perturbations of the domain. We start by terms of leading order, where we first use (\ref{Eq:13DoaminR}) e.g.
	\begin{align}
	\frac{d^+}{d\eps}\restr{\left((\nab\usol,\nab\vt)_{L^2_{\R}(\Om_\eps)}\right)}{\eps=0}=D_m(\nab\usol,\nab\vt)_{L^2_{\R}(\Om)}+div(\Vv) (\nab\usol,\nab\vt)_{L^2_{\R}(\Om)}
	\end{align}
	Then the first integral is rewritten using (\ref{Eq:12MatGradProdR})
	\begin{equation}
	\begin{aligned}
	=&(\nab D_mu,\nab \vt)_{L^2_{\R}(\Om)}-(\nab \usol,(\nab \Vv+\nab \Vv^T)\nab \vt)_{L^2_{\R}(\Om)}+
	\\&(\nab\vt,\nab D_m\usol)_{L^2_{\R}(\Om)}+div(\Vv) (\nab\usol,\nab\vt)_{L^2_{\R}(\Om)})
	\end{aligned}
	\end{equation}
	For the second integral we obtain again using (\ref{Eq:13DoaminR})
	\begin{align}
	\frac{d^+}{d\eps}\restr{-\left((k^2\usol,\vt)_{L^2_{\R}(\Om_\eps)}\right)}{\eps=0}=-D_m(k^2\usol,\vt)_{L^2_{\R}(\Om)}-div(\Vv) (k^2u,\vt)_{L^2_{\R}(\Om)}
	\end{align}
	Similar to before the first integral is rewritten using (\ref{Eq:10MatProdR}), such that we get
	\begin{align}
	=-(k^2D_m\usol,\vt)_{L^2_{\R}(\Om)}-(k^2\usol,D_m\vt)_{L^2_{\R}(\Om)}-div(\Vv)( k^2(\usol,\vt)_{L^2_{\R}(\Om)}
	\end{align}
	If we finally rearrange the terms with $D_m(u)$ and $D_m(\vt)$, let them act as test functions, apply the saddle point conditions, which means that the state equation (\ref{Eq:19Adjoint}) and adjoint equation (\ref{Eq:19Adjoint}) are fulfilled, the terms consisting $D_m(u)$ and $D_m(\vt)$ cancel. By adding all terms above up and using the definition of the inner product (\ref{Eq:CompProduct}), the shape derivative $DJ_1[\Vv]$ is established.
\end{proof}
\begin{remark}
	In case of partial reflection the boundary terms are obtained by observing
	\begin{align}
	\frac{d^+}{d\eps}\restr{\left((\alpha\usol,\vt)_{L^2_{\R}(\Gamma_{5\eps})}\right)}{\eps=0}=D_m(\alpha\usol,\vt)_{L^2_{\R}(\Gamma_{5})}+div_{\Gamma_{5}}(\Vv)(\alpha\usol,\vt)_{L^2_{\R}(\Gamma_{5})}
	\end{align}
	with the convention
	\begin{align}
	div_{\Gamma}(\Vv)=div(\Vv)-\nv\cdot(\nab \Vv)\nv
	\end{align}
	and applying saddle point conditions.
\end{remark}
Having obtained the Helmholtz shape derivative in volume form we can now deduce the shape derivate in surface form.
\begin{theorem}
	(Shape Derivative Boundary Form)
	Under the assumptions of Theorem \ref{Theo:2} the shape derivative of the objective $J_1$ (without perimeter regularization for full-reflecting boundaries) at $\Om$ in the
	direction $\Vv$ is given by
	\begin{equation}
	\begin{aligned}
	DJ_{1,\Gamma}(\Om)[\Vv]=&\Re\left[\int_{\Ge}\left(\nab v\cdot\nab\usol-k^2v\usol\right)\langle\Vv,\nv\rangle\right]
	\end{aligned}
	\label{Eq:22SD}
	\end{equation}
\end{theorem}
\begin{proof}
	The result is obtained using integration by parts on (\ref{Eq:21SD}) and vector calculus identities, we refer to \cite{Feijoo2003} for a more detailed derivation.
\end{proof}
\begin{remark}
	For a sound-soft scatterer \cite{Feijoo2003} or partially absorbing boundary conditions (\ref{PABC}), we would obtain additional terms depending on partial derivatives in normal direction.
\end{remark}
Following \cite{FeijooG2004} we are now in the position to derive the topological derivative as described in the following theorem.
\begin{theorem}
	(Topological Derivative)
	Under the assumptions of Theorem \ref{Theo:2} the topological derivative $DJ_{1,T}$ of the objective $J_1$ (without perimeter regularization for full-reflecting boundaries) is given as
	\begin{equation}
	\begin{aligned}
	DJ_{1,T}=\Re\left[\nab v\cdot\nab\usol-k^2v\usol\right]
	\end{aligned}
	\label{Eq:23TD}
	\end{equation}
\end{theorem}
\begin{proof}
	The topological derivative is obtained using the topological-shape sensitivity method, that relates topological and shape derivatives \cite{Novotny2003}\cite{Novotny2013}, i.e.
	\begin{align}
	DJ_{T}=\lim_{t\rightarrow 0^+}\left[\frac{1}{h'(t)V_n}\frac{d}{d(\delta t)}\big|_{\delta t=0}J(\Om_{t+\delta t})\right]\label{Eq:TSSM}
	\end{align}
	where $\delta>0$, $h'(t)$ s.t. $0<DJ_T<\infty$.
	Simple insertion of (\ref{Eq:22SD}) in (\ref{Eq:TSSM}) and estimates as in \cite{FeijooG2004}, such that we can define $h'(t)=2\pi t$ and $h(t)=\pi t^2$, lead to the assertion.
\end{proof}
Finally for completeness we require the shape derivative of the volume penalty and the perimeter regularization as \cite{Sokolowski1992}
\begin{align}
DJ_3&=\nu_1\int_{\Om}div(\Vv)\diff x \label{Eq:31Dcvol}\\
DJ_4&=\nu_2\int_{\Ge}\kappa\langle \Vv,\nv\rangle \diff s=\nu_2\int_{\Ge}\nab\cdot \Vv-\langle \frac{\partial \Vv}{\partial \nv},\nv\rangle \diff s\label{Eq:30DJreg}
\end{align}

\section{Numerical Results} \label{sec:NumRes}
In this section we firstly describe the numerical algorithm to solve the PDE-constrained optimization problem and present applications to a simplistic domain for transmissive and non-transmissive obstacle such as a domain representing the Langue-de-Barbarie a coastal section in the north of Senegal.
\subsection{Implementation Details}
We are relying on the classical structure of adjoint-based shape optimization algorithms gradient-descent algorithms. However, we motivate the location and shape of the initial obstacle by the usage of the topological derivative $DJ_T$ (\ref{Eq:23TD}), that means we are exploiting the obtained scalar field to initialize an obstacle with the help of a filter which is based on the density-based spatial clustering algorithm (DBSCAN) \cite{Ester1996}. The obstacle is then deformed in a second step by the the usage of shape optimization. The procedure is shortly sketched in Figure \ref{fig:ShapeOptAlgo}.
\begin{figure}[htb!]
	\begin{algorithm}[H]
		\SetAlgoLined
		Evaluate Topological Derivative $DJ_{\{1,2\},T}$ \& initialize Obstacle via DBSCAN\\
		\While{$||DJ_{\{1,2\}}(\Om_k)||>\eps$}{
			2 Calculate State $u_k$\\
			3 Calculate Adjoint $v_k$\\
			4 Use $DJ_{\{1,2\},3,4}(\Om_k)[\Vv]$ to calculate Gradient $W_k$\\
			5 Perform Linesearch for $\tilde{W}_k$\\
			6 Deform $\Om_{k+1}\longleftarrow \mathcal{D}_{\tilde{W}_k}(\Om_k)$
		}
		\caption{Shape Optimization Algorithm}
	\end{algorithm}
	\caption{Shape Optimization Algorithm}
	\label{fig:ShapeOptAlgo}
\end{figure}

To compute the solution to the boundary value problem (\ref{Helmholtz})-(\ref{Sommerfeld}), the adjoint problem (\ref{Eq:19Adjoint}) and to finally deform the domain we are relying on the finite element solver FEniCS \cite{FEniCS2015}. High accuracy for the elliptic, constraining PDE is achieved using a Continuous Galerkin (CG) method of order $p\ge1$ to discretize in space.
The mentioned deformation or update of the finite element mesh in each iteration is done via the solution $W:\Om\rightarrow\R^2$ of the linear elasticity equation, which stems from the usage of the Steklov-Poincaré metric \cite{Schulz2016}
\begin{equation}
\begin{aligned}
\int\sigma(\Wv):\eps(\Vv)&=DJ[\Vv] \hspace{1cm} \quad &\forall\Vv\in H_0^1(\Om,\R^2)\\
\sigma:&=\lambda Tr(\eps(\Wv))I+2\mu\eps\\
\eps(\Wv):&=\frac{1}{2}(\nab \Wv+\nab \Wv^T)\\
\eps(\Vv):&=\frac{1}{2}(\nab \Vv+\nab \Vv^T)
\end{aligned}
\label{Eq:29LinearElasticity}
\end{equation}
where $\sigma$ and $\eps$ are called strain and stress tensor and $\lambda$ and $\mu$ are called Lame parameters. In our calculations we have chose $\lambda=0$ and $\mu$ as the solution of the following Poisson Problem 
\begin{equation}
\begin{aligned}
-\bigtriangleup\mu&=0 \hspace{1cm} &\text{in }\quad &\Om&&\\
\mu&=\mu_{max} \hspace{1cm} &\text{on } \quad&\Ge&&\\
\mu&=\mu_{min} \hspace{1cm} &\text{on }\quad &\Ga, \Gb,\Gc,\Gd&&\\
\end{aligned}
\label{Eq:33Lame}
\end{equation}
The source term $DJ[\Vv]$ in (\ref{Eq:29LinearElasticity}) consists of a volume and surface part, i.e. $DJ[\Vv]=DJ_\Om[\Vv] + DJ_{\Ge}[\Vv]$. 
Here the volumetric share comes from our Helmholtz shape derivative $DJ_{\{1,2\}}$ and the shape derivative volume penalty $DJ_3$, where we only assemble for test vector fields whose support intersects with the interface $\Ge$ and is set to zero for all other basis vector fields. 
The surface part comes from the shape derivative parameter regularization $DJ_4$.
\begin{remark}
	In order to guarantee the attainment of useful shapes, which minimize the objective, a backtracking line search is used, which limits the step size in case the shape space is left \cite{Welker2017} i.e. having intersecting line segments or the objective is non-decreasing.
\end{remark}
\begin{remark}
	Solutions to state, adjoint and shape derivative, require the manual calculation of the real part by splitting all occurring trial and test functions in real and imaginary part.
\end{remark}

\subsection{Ex.1: The Simplistic Mesh}\label{sec:exhalfcircled}
In the first example, we will look at the model problem, that was described in Section \ref{HelmModForm}. We interpret $\Ga$ and $\Ge$ as the reflective coastline and obstacle, $\Gb$ and $\Gc$ as the lateral, such as $\Gd$ as the open sea boundary. As it is described before in Section \ref{sec:NumRes}, the topological derivative can be used in line with a filter to determine the location of an initial obstacle. Exemplifying, we show in Figure \ref{fig:TopoDer} an initial field on the simplistic mesh for which the topological derivative is calculated. A filter in form of a DBSCAN-algorithm is then used in Figure \ref{fig:DBSCAN} to initialize an obstacle. Here results are shown for a minimum number of points in a cluster $m=10$ and threshold $\eps=7$, where colours apart from blue build useful clusters.
\begin{figure}[htb!]
	\begin{center}
		\includegraphics[scale=0.2]{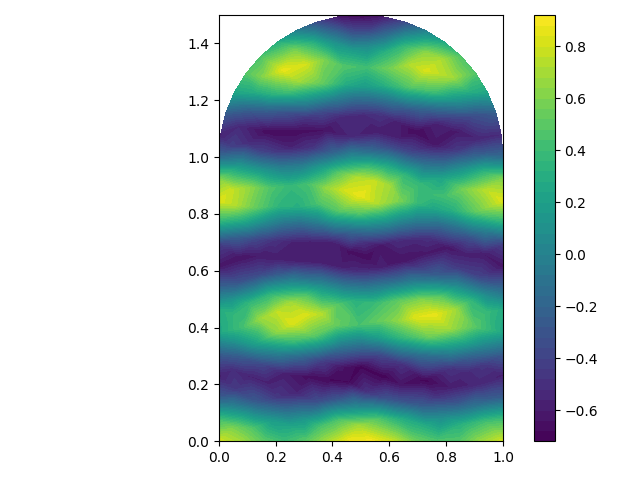}
		\includegraphics[scale=0.2]{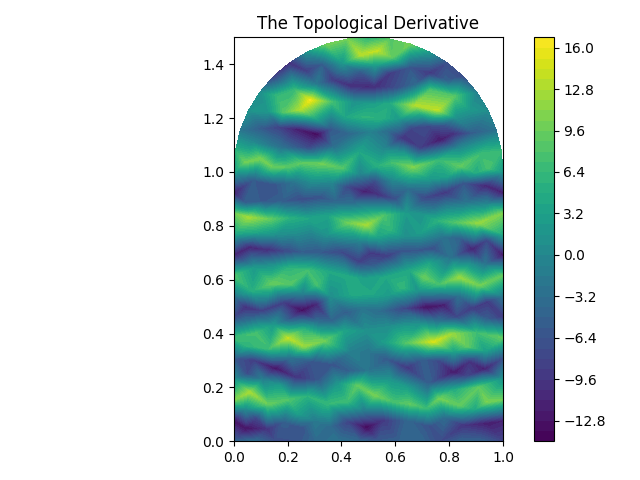}
		\caption[Field for No Obstacle \& Topological Derivative]{1. Field for no Obstacle, 2. Topological Derivative}
		\label{fig:TopoDer}
	\end{center}
\end{figure}

\begin{figure}[htb!]
	\begin{center}
		\includegraphics[scale=0.2]{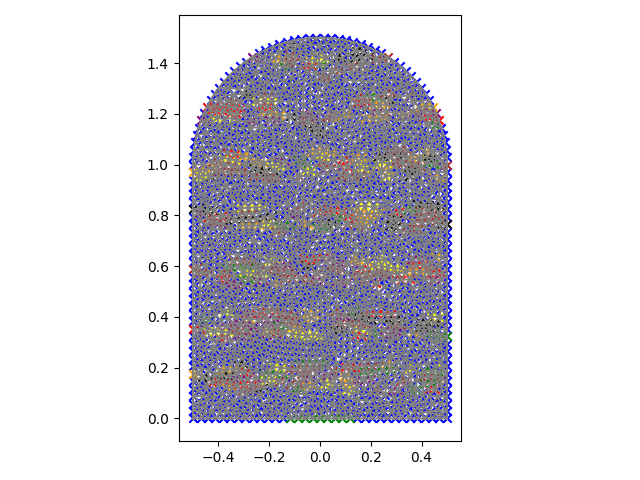}
	\end{center}
	\caption{Solution to DBSCAN}
	\label{fig:DBSCAN}
\end{figure}
We have used this information to generate the meshes in Figure \ref{fig:Ex1Initial Meshes} with the the mesh generator GMSH \cite{Geuzaine2009}. We have discretized finer around the obstacle to ensure a high resolution for the shape optimization routine. A reducing effect along the shore $\Ga$ of the created meshes for single and multi-wave case can already be observed for the pure placement of the obstacles in Figure \ref{fig:Ex1InitialFields}.  The forthcoming analysis is based on wave description (\ref{Helmholtz}), we firstly model a single wave perpendicular to the obstacle's lower boundary by choosing $\phi=1.5\pi$ and a suitable wave number e.g. $k=12$ in the first two figures of Figures \ref{fig:Ex1InitialFields} and \ref{fig:Ex1FinalFields}. In the multi wave case (\ref{Eq:ObjectiveSum}) we model the sum of $N=3$ waves with $\phi_j\in\{1.25\pi,1.5\pi,1.75\pi\}$ for weights $w_j\in\{0.5,0.4,0.1\}$ such as $M=2$ frequencies with $k_i\in\{11,15\}$, which can be taken from the third figures in Figures \ref{fig:Ex1InitialFields}, \ref{fig:Ex1FinalFields} and can be interpreted as strong waves from north-west. In all the test cases we model full reflection, i.e. $\alpha=0$. In the objective we enforce regularization of the perimeter by a weight of $\nu_2=0.1$. The solution to the state and adjoint equation is of linear nature, hence we use the FEniCS solver for linear partial differential equations. Having solved state and adjoint equations the mesh deformation is performed as described before, where we specify $\mu_{min}=10$ and $\mu_{max}=100$ in (\ref{Eq:33Lame}). The step size is at $\rho=0.04$ and shrinks whenever criteria for line searches are not met.
We can extract results of shape optimization from Figure \ref{fig:Ex1FinalFields}, where we due to Figure \ref{fig:Ex1Objective} have obtained a significant decrease in the objective values.
\begin{figure}[!htbp]
	\centering
	\begin{tikzpicture}
	\node[anchor=south west,inner sep=0] (1) 
	{\includegraphics[scale=0.21]{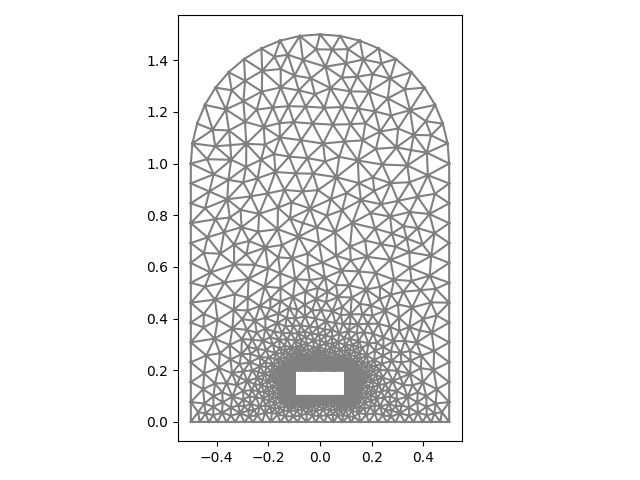}};
	\vspace{-3cm}
	\node[right= 0.01cm of 1](2)
	{\includegraphics[scale=0.21]{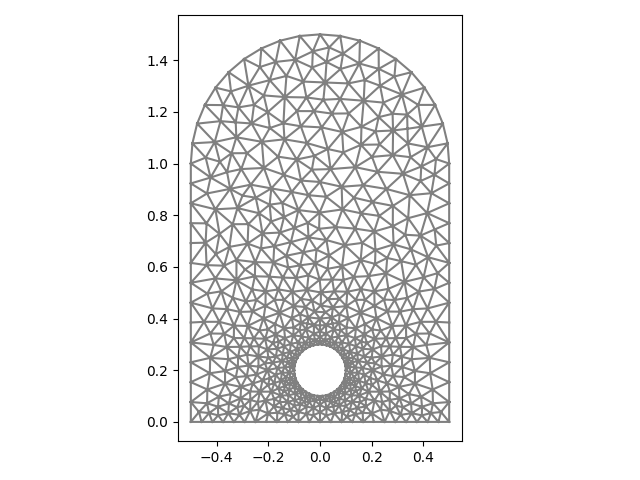}};
	\end{tikzpicture}
	\caption[Ex.1 Initial Meshes]{1.: Initial Mesh for Rectangular Obstacle, 2.:  Initial Mesh for Circled Obstacle}
	\label{fig:Ex1Initial Meshes}
\end{figure} 
\begin{figure}[!htbp]
	\centering
	\begin{tikzpicture}
	\node[anchor=south west,inner sep=0] (1) 
	{\includegraphics[scale=0.21]{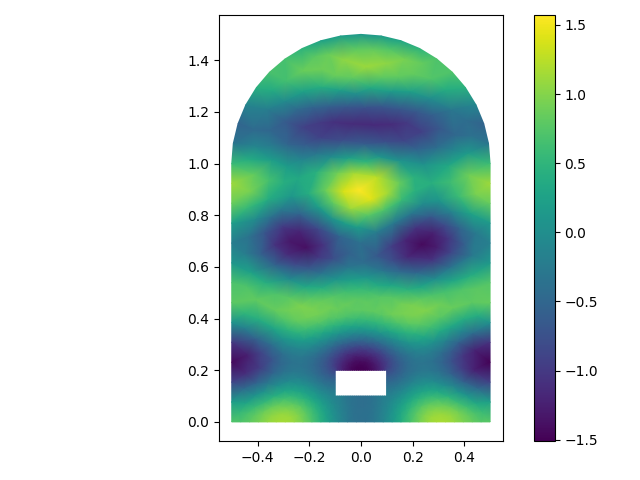}};
	\node[right= 0.01cm of 1](2)
	{\includegraphics[scale=0.21]{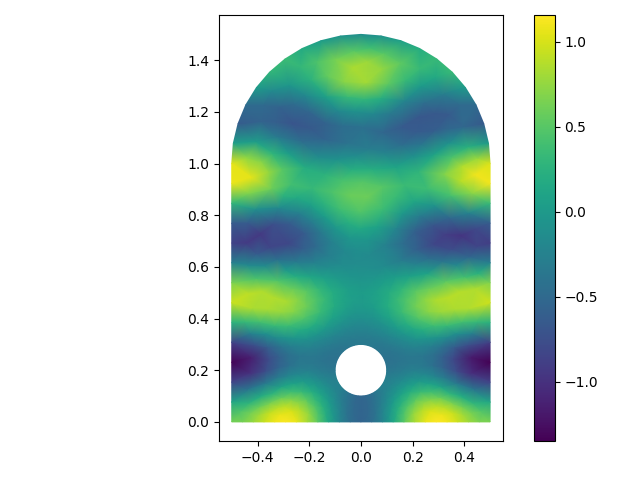}};
	\node[right = 0.01cm of 2](3)
	{\includegraphics[scale=0.21]{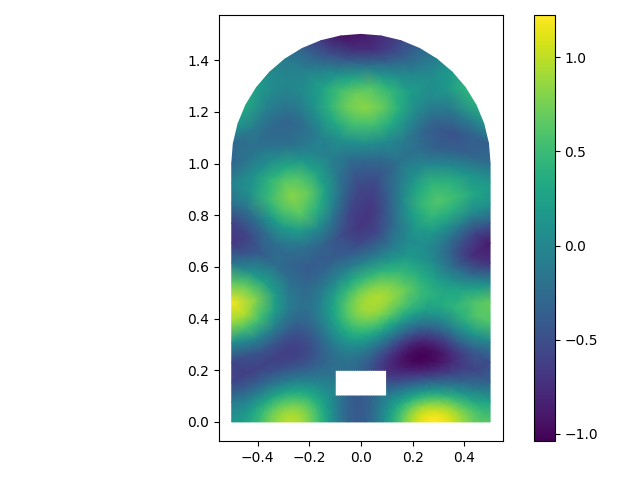}};
	\end{tikzpicture}
	\caption[Ex.1 Initial Fields]{1.: Initial Field for Rectangular Obstacle, 2.:  Initial Field for Circled Obstacle, 3.: Initial Field for Multiple Waves and Rectangle Obstacle}
	\label{fig:Ex1InitialFields}
\end{figure}
\begin{figure}[!htbp]
	\centering
	\begin{tikzpicture}
	\node[anchor=south west,inner sep=0] (1) 
	{\includegraphics[scale=0.21]{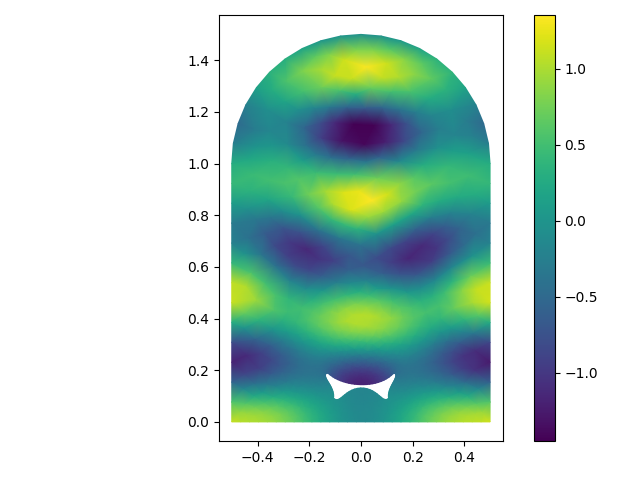}};
	\node[right= 0.01cm of 1](2)
	{\includegraphics[scale=0.21]{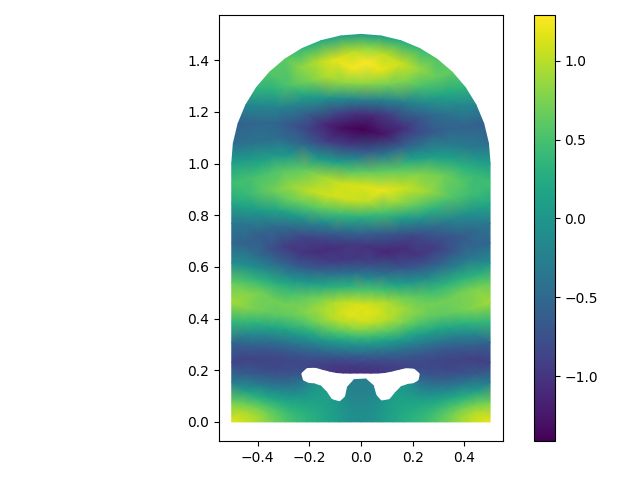}};
	\node[right = 0.01cm of 2](3)
	{\includegraphics[scale=0.21]{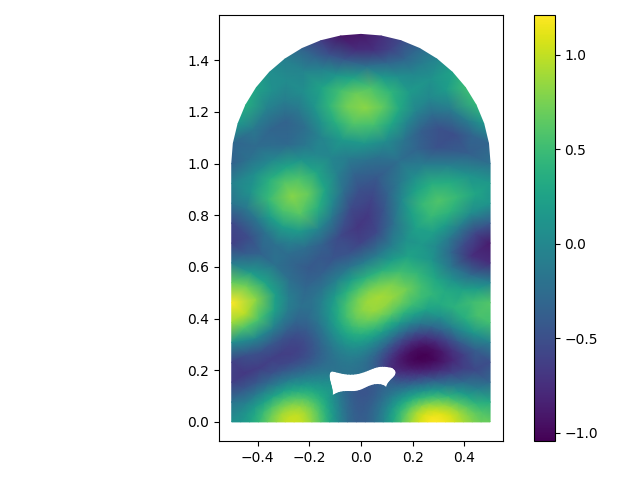}};
	\end{tikzpicture}
	\caption[Ex.1 Optimized Fields]{1.: Optimized Field for Rectangular Obstacle, 2.:  Optimized Field for Circled Obstacle, 3.: Optimized Field for Two Circled Obstacles }
	\label{fig:Ex1FinalFields}
\end{figure}
\begin{figure}[!htbp]
	\centering
	\begin{tikzpicture}
	\hspace{0.5cm}
	\node[anchor=south west,inner sep=0] (1) 
	{\includegraphics[scale=0.17]{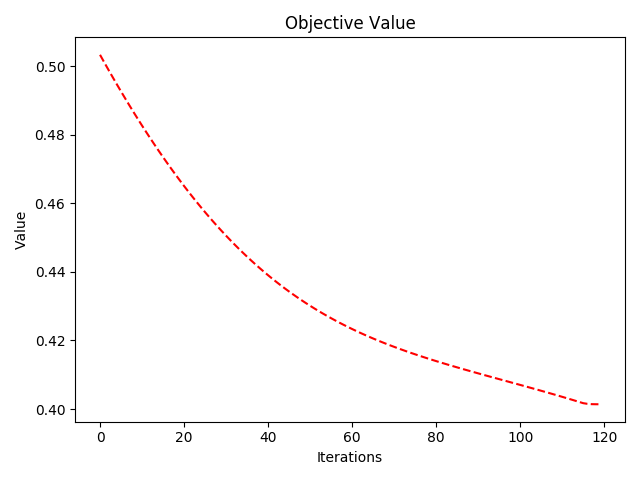}};
	\node[right= 0.01cm of 1](2)
	{\includegraphics[scale=0.17]{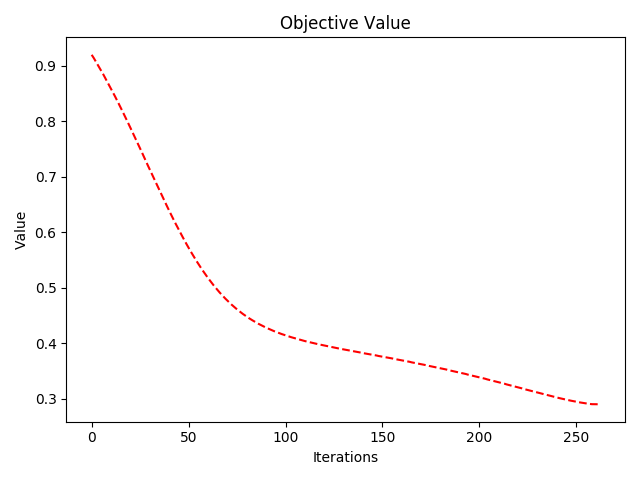}};
	\node[right = 0.01cm of 2](3)
	{\includegraphics[scale=0.17]{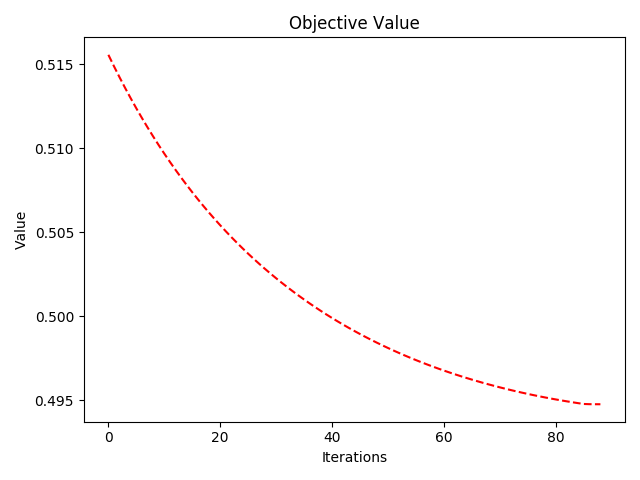}};
	\end{tikzpicture}
	\caption[Ex.1 Objectives]{1.: Objective for Rectangular Obstacle, 2.:  Objective for Circled Obstacle, 3.: Objective for Two Circled Obstacles }
	\label{fig:Ex1Objective}
\end{figure}
We now turn towards the transmissive obstacle case, i.e. (\ref{HelmholtzOm})-(\ref{Eq:HelmholtztransmissiveBC2}) is used. Wave-settings equal to the mono-wave from before but working with a discontinuous porosity coefficient for $\phi_1=1$ and $\phi_2=0.1$, that is pictured in Figure \ref{fig:TransmFigures}. We remark that for an appropriate evaluation of the weak form (\ref{HelmModForm}) the discontinuity must be inline with the positioning of the mesh nodes. We then once more target to track a specified height of field $\usol$ (cf. to Figure \ref{fig:TransmFigures}, 1., left). The optimized form, which can be seen in Figure \ref{fig:TransmResults}, is especially notable. It forms two lines of vertically shifted stretched obstacles, which is due to the initial placement of the obstacle, participating on two increased fronts in the scalar field of the topological derivative (cf. to Figure \ref{fig:TopoDer}). The shape is then minimized in areas of low values for the mentioned scalar field. The final shape was reached after $13765$ iterations, where the norm of the gradient reached the convergence threshold. The convergence of the objective function can be taken from Figure \ref{fig:TransmResults}.
\begin{figure}[!htb]
	\centering
	\begin{tikzpicture}
	\node[anchor=south west,inner sep=0] (1) 
	{\includegraphics[scale=0.2]{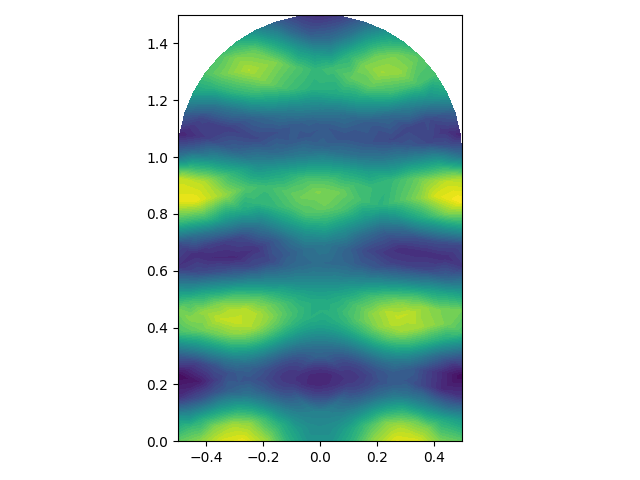}};
	
	\node[right= 0.01cm of 1](2)
	{\includegraphics[scale=0.2]{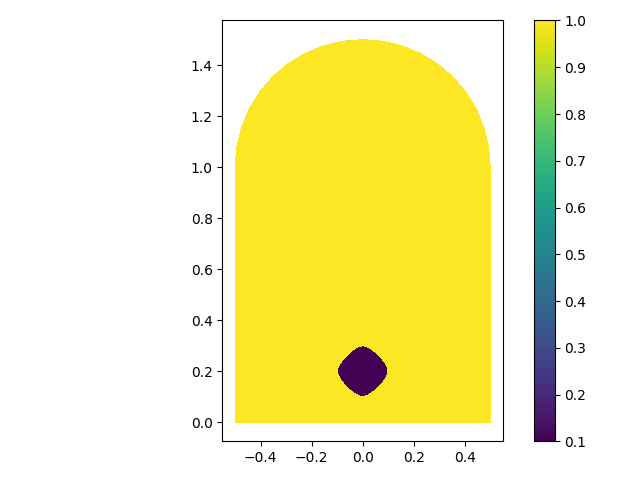}};
	\end{tikzpicture}
	\caption[Ex.1 Transmissive Setting]{1.: Initial Mesh for Smoothed Rectangular Obstacle, 2.: Discontinuous Permeability Coefficient $\phi$}
	\label{fig:TransmFigures}
\end{figure} 
\FloatBarrier
\begin{figure}[!htb]
	\centering
	\begin{tikzpicture}
	\node[below= 0.01cm of 1](3)
{\includegraphics[scale=0.2]{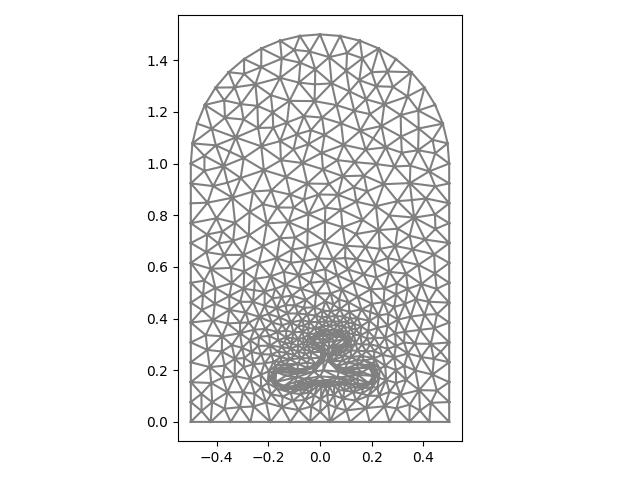}};
\node[right= 0.01cm of 3](4)
{\includegraphics[scale=0.2]{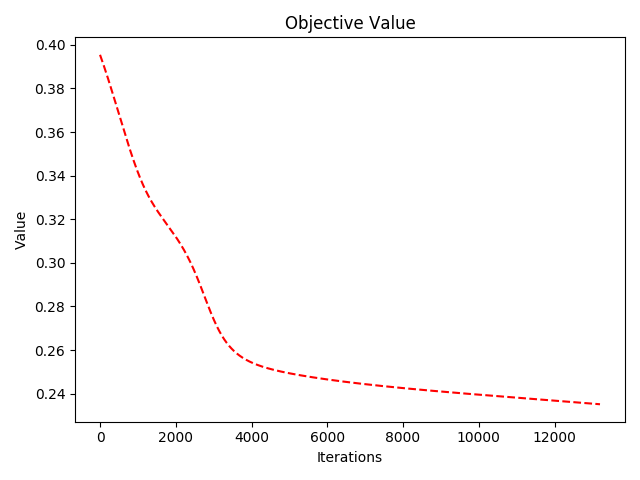}};
\end{tikzpicture}
	\caption[Ex.1 Transmissive Results]{ 1.: Optimized Mesh, 2.: Objective Function}
\label{fig:TransmResults}
\end{figure} 
\subsection{Ex.2: Langue-de-Barbarie}
A more realistic computation is performed in the second example. Here we look at the Langue-de-Barbarie a coastal section in the north of Dakar, Senegal. In 1990 it consisted of a long offshore island, which eroded in three parts within two decades. Waves now travel unhindered to the mainlands, which causes severe damage and destroyed large habitats.
Adjusting our model to this specific coastal section starts on mesh level. Shorelines are taken from the free GSHHG databank\footnote[2]{https://www.ngdc.noaa.gov/mgg/shorelines/} following \cite{Avdis2016}. We build an interface from a geographical information system (QGIS3) for processing the data to the Computer Aided Design software GMSH for the mesh generation.  Similar to the preceding example, we interpret $\Ga,\Gd$ as coastline for islands and mainlands such as the open sea boundary. We have inserted a smaller island, which shape is to be optimized in front of the second and third with boundary denoted as $\Ge$ (cf. to Figure \ref{fig:ShapeOptLDBFR},\ref{fig:ShapeOptLDBPR}). The waves
propagation is modelled mono-directionally to the shores with $\phi=1.8\pi$ and $k=35$. The initial mesh can be extracted from Figure \ref{fig:WavePropagationLDBMesh}.
\begin{figure}[!htbp]
	\centering
	\begin{tikzpicture}
	\node[anchor=south west,inner sep=0] (1) 
	{\includegraphics[width=4.5cm,height=3.5cm,scale=0.5]{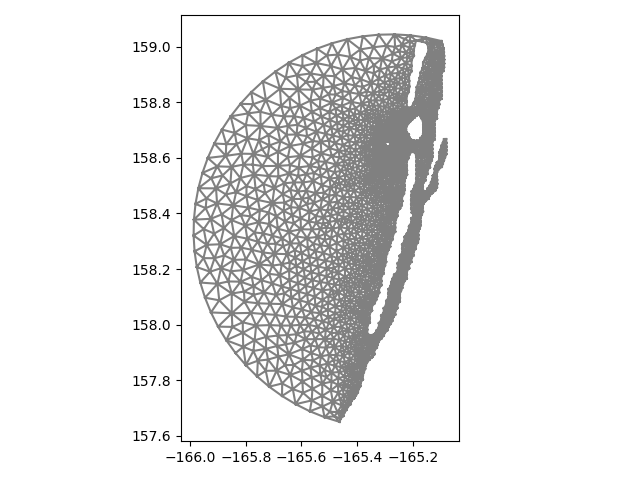}};
	\end{tikzpicture}
	\caption[Ex.2 Initial Mesh]{Initial LdB Mesh}
	\label{fig:WavePropagationLDBMesh}
\end{figure}
 Figures \ref{fig:ShapeOptLDBFR} and \ref{fig:ShapeOptLDBPR} picture fields to the optimized meshes for $\alpha_1=0$ and $\alpha_1=0.2$ using (\ref{Eq:Isaacsonalpha}) and an initial step size $\rho=0.01$ after $177$ and $39$ steps of optimization.
\begin{figure}[!htb]
	\centering
	\begin{tikzpicture}
	\node[anchor=south west,inner sep=0] (1) 
	{\includegraphics[width=4.5cm,height=3.5cm,scale=0.5]{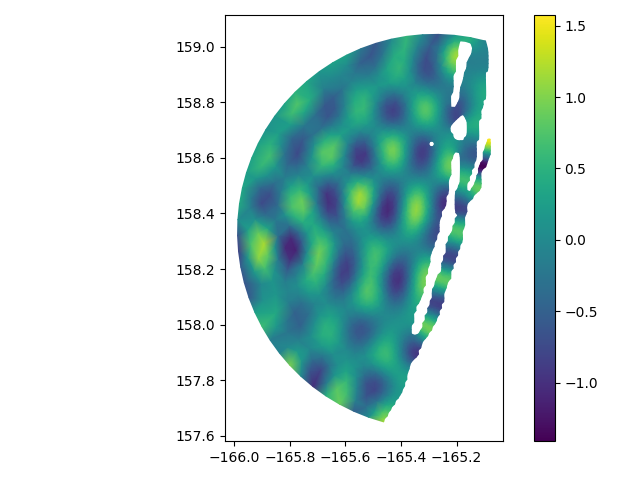}};
	\node[right= 2cm of 1](2)
	{\includegraphics[width=4.5cm,height=3.5cm,scale=0.5]{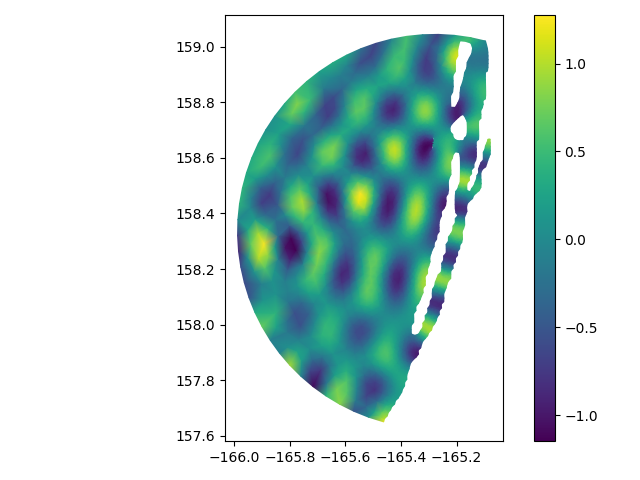}};
	
	\node (7) at (2.56,3.5)
	{\tiny Initial Field};
	\node (8) at (9.2,3.5)
	{\tiny Optimized Field};
	\draw[draw=red] (9.6,2.35) rectangle ++(0.2,0.2);
	\draw[->, >=stealth', shorten >=1pt, thick, draw= red] (9.6,2.35)   -- (11.19,2.35);
	\node (8) at (12.7,2.8)
	{\includegraphics[width=3cm,height=3cm,scale=0.5]{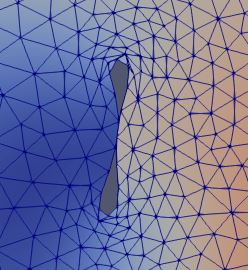}};
	
	\draw[draw=red] (2.97,2.35) rectangle ++(0.2,0.2);
	\draw[->, >=stealth', shorten >=1pt, thick, draw= red] (3,2.35)   -- (4.5,2.35);
	\node (9) at (6,2.8)
	{\includegraphics[width=3cm,height=3cm,scale=0.5]{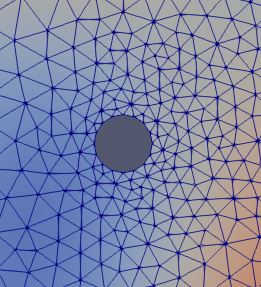}};
	\end{tikzpicture}
	\caption[Ex.2 Initial \& Optimized Field \& Obstacle]{1.: Initial Field and Obstacle for Full Reflection,
		2.: Optimized Field and Obstacle for Full Reflection}
	\label{fig:ShapeOptLDBFR}
\end{figure}
\begin{figure}[!htb]
	\centering
	\begin{tikzpicture}
	\node[anchor=south west,inner sep=0] (1) 
	{\includegraphics[width=4.5cm,height=3.5cm,scale=0.5]{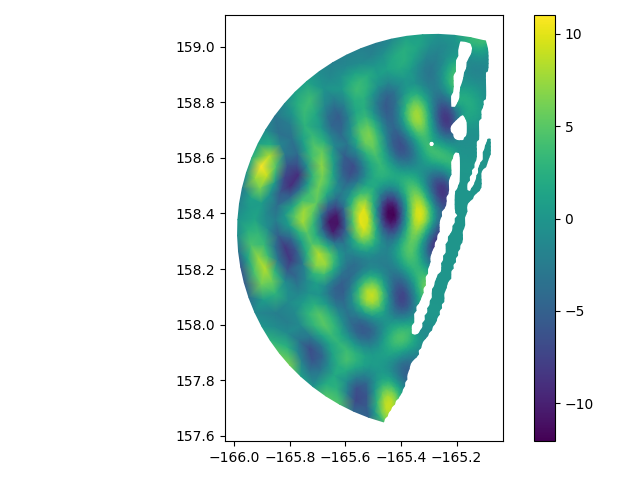}};
	\node[right= 2cm of 1](2)
	{\includegraphics[width=4.5cm,height=3.5cm,scale=0.5]{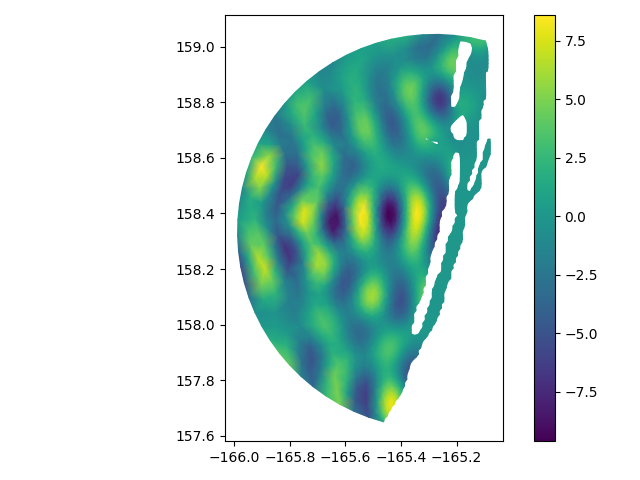}};
	
	\node (7) at (2.56,3.5)
	{\tiny Initial Field};
	\node (8) at (9.2,3.5)
	{\tiny Optimized Field};
	\draw[draw=red] (9.6,2.35) rectangle ++(0.2,0.2);
	\draw[->, >=stealth', shorten >=1pt, thick, draw= red] (9.6,2.35)   -- (11.19,2.35);
	\node (8) at (12.7,2.8)
	{\includegraphics[width=3cm,height=3cm,scale=0.5]{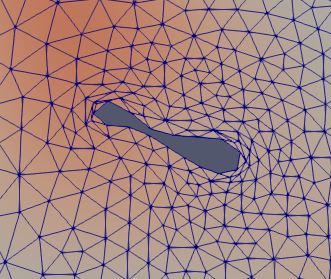}};
	
	\draw[draw=red] (2.97,2.35) rectangle ++(0.2,0.2);
	\draw[->, >=stealth', shorten >=1pt, thick, draw= red] (3,2.35)   -- (4.5,2.35);
	\node (9) at (6,2.8)
	{\includegraphics[width=3cm,height=3cm,scale=0.5]{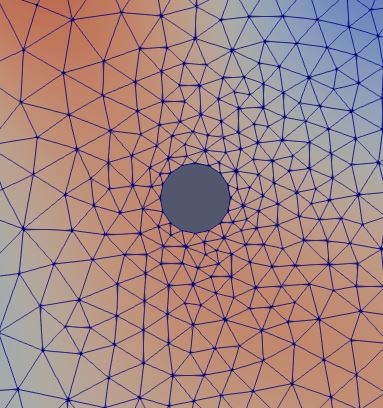}};
	\end{tikzpicture}
	\caption[Ex.2 Initial \& Optimized Field \& Obstacle]{1.: Initial Field and Obstacle for Partial Reflection,
		2.: Optimized Field and Obstacle for Partial Reflection}
	\label{fig:ShapeOptLDBPR}
\end{figure}

One can observe a similar behaviour as in Subsection \ref{sec:exhalfcircled}, where the obstacle is stretched to protect an as large as possible area. The computation stopped after obtaining intersecting line segments at the obstacle's centre in Figure \ref{fig:ShapeOptLDBFR} and reaching the convergence threshold for the norm of the gradient in Figure \ref{fig:ShapeOptLDBPR}. However, in Figures \ref{fig:TargetFunctionLDBFR} and \ref{fig:TargetFunctionLDBPR} we can still note a significant decrease in the target functional for both cases, taking into account that the area of the scatterer is comparably low to the area of the shorelines. 

\begin{figure}[!htb]
	\centering
	\begin{tikzpicture}
	\node[anchor=north west,inner sep=0] (6)  {\includegraphics[scale=0.17]{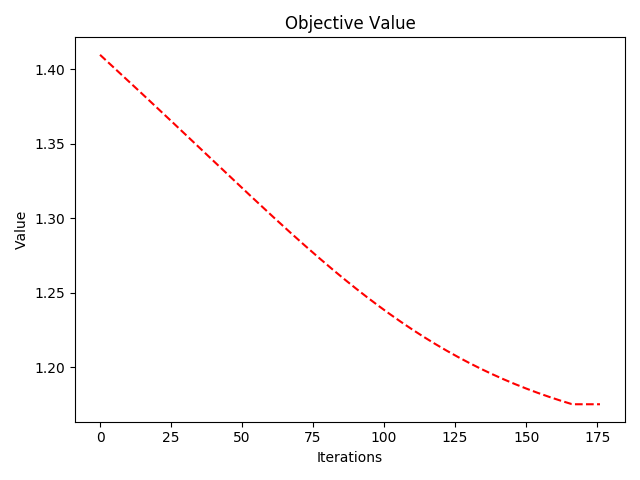}};
	\end{tikzpicture}
	\caption[Ex.2 Target Functional]{LdB Target Functional for Full Reflection}
	\label{fig:TargetFunctionLDBFR}
\end{figure}
\begin{figure}[!htb]
	\centering
	\begin{tikzpicture}
	\node[anchor=north west,inner sep=0] (6)  {\includegraphics[scale=0.17]{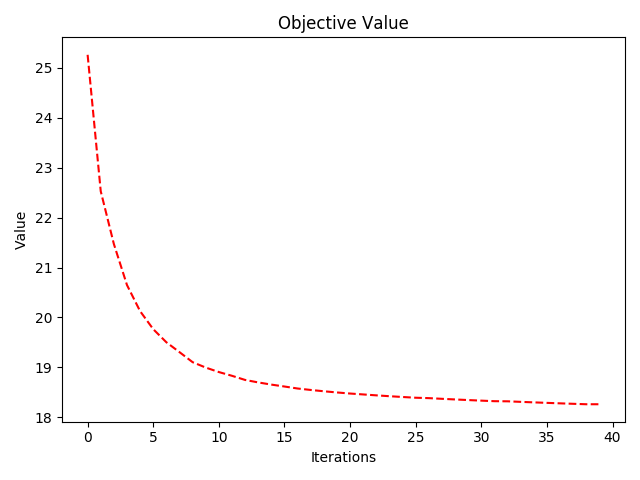}};
	\end{tikzpicture}
	\caption[Ex.2 Target Functional]{LdB Target Functional for Partial Reflection}
	\label{fig:TargetFunctionLDBPR}
\end{figure}
\begin{remark}
	From a practical, durability standpoint thin obstacles as in Figures \ref{fig:Ex1FinalFields} and \ref{fig:ShapeOptLDBFR}, which also lead to breakdown of the optimization algorithm, due to intersecting line segments, are not ideal. To circumvent this problem we have used a thinness penalty (cf. to \cite{Allaire2016}) in works dealing with Shallow Water Equations \cite{Schlegel20212}\cite{Schlegel20213}, but this should not be treated here.
\end{remark}

\section{Conclusion}
We have derived the stationary continuous adjoint and shape derivative  in volume form for the Helmholtz equation with suitable boundary conditions. The results were tested on a 2-dimensional simplistic domain and a more comprehensive one picturing the Langue-de-Barbarie coastal section. The optimized shape strongly orients itself to wave directions and to the respective coastal section that is to be protected. The results can be easily adjusted for arbitrary meshes, objective functions and different wave such as boundary properties.

\section*{Acknowledgement}
This work has been supported by the Deutsche
Forschungsgemeinschaft within the Priority program SPP 1962 "Non-smooth and Complementarity-based Distributed Parameter Systems: Simulation and Hierarchical Optimization". The authors would like to thank Diaraf Seck (Université Cheikh Anta Diop, Dakar, Senegal) and Mame Gor Ngom (Université Cheikh Anta Diop, Dakar, Senegal) for helpful and interesting discussions within the project Shape Optimization Mitigating Coastal Erosion (SOMICE). 

\bibliographystyle{unsrt}
\bibliography{bibliography}  






\end{document}